\documentclass[a4paper,11pt]{amsart}

\usepackage{amsmath}
\usepackage{amsthm}
\usepackage{amssymb}
\usepackage{mathptm}

\theoremstyle{plain}
\newtheorem{thm}{Theorem}[section]

\newtheorem{lem}[thm]{Lemma}

\newtheorem{prop}[thm]{Proposition}
\theoremstyle{definition}
\newtheorem{defn}[thm]{Definition}

\theoremstyle{remark}

\title{$1$-primitive near-rings}

\author{Gerhard Wendt}
\address{Institut f\"ur Algebra, Johannes Kepler Universit\"at Linz}
\email{Gerhard.Wendt@jku.at}

\urladdr{http://www.algebra.uni-linz.ac.at}
\thanks{This work has been supported by grant P23689-N18 of the Austrian National Science Fund (FWF)}
\begin{document}
\setlength{\parindent}{0pt}
\maketitle
\begin{abstract} We combine the concept of sandwich near-rings with that of centralizer near-rings to get a classification of zero symmetric $1$-primitive near-rings as dense subnear-rings of centralizer near-rings with sandwich multiplication. This result generalizes the well known density theorem for zero symmetric $2$-primitive near-rings with identity to the much bigger class of zero symmetric $1$-primitive near-rings not necessarily having an identity.
\end{abstract}
\section{Introduction}
We consider right near-rings, this means the right distributive law holds, but not necessarily the left distributive law. The notation is that of \cite{Pilz1}. Primitive near-rings play the same role in the structure theory of near-rings as primitive rings do in ring theory. When a primitive near-ring happens to be a ring, then it is a primitive ring in the usual sense. However, in near-ring theory there exist several types of primitivity depending on the type of simplicity of a near-ring group. A complete description of so called $2$-primitive near-rings with identity is available. Such near-rings are dense subnear-rings of special types of centralizer near-rings (see \cite{Pilz1} for a thorough discussion). If one studies primitive near-rings without necessarily having an identity, then still results are available but much more technical in detail in comparison to the $2$-primitive case with identity. Combination of the concepts of centralizer near-rings and sandwich near-rings were used in \cite{Wendt3} and in \cite{FP} to describe $1$-primitive near-rings which do not necessarily have an identity. In this paper we will extend and simplify the results obtained in \cite{Wendt3} and \cite{FP} and we will also consider $2$-primitive near-rings as a special case of $1$-primitive near-rings.

We will now briefly give the definitions for zero symmetric primitive near-rings to settle our notation which we will keep throughout the paper.

Let $N$ be a zero symmetric near-ring. Let $\Gamma$ be an $N$-group of the near-ring $N$. An $N$-ideal $I$ of $\Gamma$ is a normal subgroup of the group $(\Gamma, +)$ such that $\forall n \in N \forall \gamma \in \Gamma \forall i \in I: n(\gamma + i)-n\gamma \in I$. When $N$ is itself considered as an $N$-group, then an $N$-ideal of $N$ is just a left ideal of the near-ring. An $N$-group $\Gamma$ of the near-ring $N$ is of type $0$ if $\Gamma \neq \{0\}$,  if there are no non-trivial $N$-ideals in $\Gamma$, so $\Gamma$ is a simple $N$-group, and if there is an element $\gamma \in \Gamma$ such that $N\gamma =\Gamma$. Such an element $\gamma$ will be called a generator of the $N$-group $\Gamma$. The $N$-group $\Gamma$ is of type $1$ if it is of type $0$ and moreover we have $N\gamma=\Gamma$ or $N\gamma=\{0\}$ for any $\gamma \in \Gamma$. Let $K$ be a subgroup of the $N$-group $\Gamma$. $K$ is called an $N$-subgroup of $\Gamma$ if $NK\subseteq K$. The $N$-group $\Gamma$ is called $N$-group of type $2$ if $N\Gamma \neq \{0\}$ and there are no non-trivial $N$-subgroups in $\Gamma$. It is easy to see that an $N$-group of type $2$ is also of type $1$. In case $N$ has an identity element, an $N$-group of type $1$ is also of type $2$ (see \cite{Pilz1}, Proposition 3.4 and Proposition 3.7).

Given an $N$-group $\Gamma$ and a subset $\triangle \subseteq \Gamma$ then $(0:\triangle)=\{n \in N | \forall \gamma \in \triangle : n\gamma =0\}$ will be called the annihilator of $\triangle$. $\Gamma$ will be called faithful if  $(0:\Gamma)=\{0\}$. 

A near-ring $N$ is called $1$-primitive if it acts on a faithful $N$-group $\Gamma$ of type $1$. In such a situation we will say that the near-ring acts $1$-primitively on the $N$-group $\Gamma$.

It is a well known fact that any zero symmetric and $1$-primitive near-ring with identity which is not a ring is dense (i.e. equal to, in the finite case) in a centralizer near-ring $M_S(\Gamma):=\{f: \Gamma \rightarrow \Gamma | \forall \gamma \in \Gamma\ \forall s \in S: s(f(\gamma)) = f(s(\gamma))\ \text{and}\ f(0)=0\}$, where $(\Gamma, +)$ is a group, $0$ denoting its neutral element w.r.t. $+$, $S$ is a fixedpointfree automorphism group of $\Gamma$ and the near-ring operations are the pointwise addition of functions and function composition (see \cite{Pilz1}, Theorem $4.52$ for a detailed discussion).

In case a $1$-primitive near-ring contains no identity element, the situation gets more complicated and one can use sandwich near-rings to get a classification of $1$-primitive near-rings as certain kind of centralizer near-rings. We introduce the concept of a sandwich near-ring in the next definition which we will need in the next sections. The operation symbol $\circ$ stands for composition of functions.

\begin{defn}\label{Definition1}
Let $(\Gamma,+)$ be a group, $X \subseteq \Gamma$ a subset of $\Gamma$ containing the zero $0$ of $(\Gamma,+)$  and $\phi:\Gamma \longrightarrow X$ a map such that $\phi(0)=0$. Define the following operation $\circ'$ on $\Gamma^X$ ($\Gamma^X$ denoting the set of all functions mapping from $X$ to $\Gamma$): $f\circ' g := f \circ \phi \circ g$ for $f, g \in \Gamma^X$. Then $(\Gamma^X,+,\circ')$ is a zero symmetric near-ring. Let  $M_0(X,\Gamma,\phi):=\{f: X \rightarrow \Gamma | f(0)=0\}$. With respect to the operations $+$ and $\circ'$, $M_0(X,\Gamma,\phi)$ is a zero symmetric subnear-ring of $(\Gamma^X,+,\circ')$. We call  $M_0(X,\Gamma,\phi)$ a sandwich near-ring and $\phi$ is called the sandwich function.
\end{defn}
We give two examples of such sandwich near-rings in the following. They will also serve as examples explaining the concepts of primitivity by doing concrete calculations. Let $N:=\{f \in M_0(\Bbb{Z}_4)|f(2)=f(3)=0\}$. With respect to pointwise addition of functions and function composition $N$ is a zero symmetric near-ring which acts faithfully on the $N$-group $\Bbb{Z}_4$. For $\gamma \in \{0,2,3\}$ we have $N\gamma =\{0\}$ and if $\gamma = 1$ we have  $N\gamma=\Bbb{Z}_4$. $\{0,2\}$ is not an $N$-ideal of the $N$-group. To see this, let $f \in N$ be such that $f(1)=3$. Then, $f(1+2)-f(1)=1 \not \in \{0,2\}$. So, $N$ acts $1$-primitively on $\Bbb{Z}_4$. Clearly, $\{0,2\}$ is an $N$-subgroup of the $N$-group $\Bbb{Z}_4$. Thus, $N$ does not act $2$-primitively on $\Bbb{Z}_4$. Note that $N$ has $4$ elements and cannot be isomorphic to a centralizer near-ring since it is missing an identity element. Now let $X:=\{0,1\}$ and $\phi: \Bbb{Z}_4 \rightarrow X$ such that $\phi(0)=\phi(2)=\phi(3)=0$ and $\phi(1)=1$. Then, $M_0(X,\Bbb{Z}_4,\phi)$ is a sandwich near-ring. For $f \in M_0(X,\Bbb{Z}_4,\phi)$ we have $f(0)=0$ and $f(1) \in \{0,1,2,3\}$. Thus, $M_0(X,\Bbb{Z}_4,\phi)$ has $4$ elements. Let $g_1, g_2 \in M_0(X,\Bbb{Z}_4,\phi)$ such that $g_1(1)=1$ and $g_2(1)=2$. Then multiplication $\circ'$ is done as $g_1 \circ' g_2=g_1 \circ \phi \circ g_2$. So,  $g_1 \circ' g_2 (1)=g_1(\phi(g_2(1)))=g_1(\phi(2))=g_1(0)=0$. Thus $g_1 \circ' g_2$ is the zero function. Let $f \in N$ and let $\psi_f : X \rightarrow \Bbb{Z}_4, x \mapsto f(x)$. The function $h: N \rightarrow M_0(X,\Bbb{Z}_4,\phi), f \mapsto \psi_f$ is easily seen to be a near-ring isomorphism.

Similary, if we let $N_1:=\{f \in M_0(\Bbb{Z}_4)|f(3)=0\}$ then one sees that $N_1$ is $2$-primitive on $\Bbb{Z}_4$ because $\{0,2\}$, the only non-trivial subgroup of $\Bbb{Z}_4$, is not an $N_1$-subgroup. $N_1$ has $16$ elements. Let $X=\{0,1,2\}$ and $\phi: \Bbb{Z}_4 \rightarrow X$ such that $\phi(0)=\phi(3)=0$, $\phi(1)=1$ and $\phi(2)=2$. Then, $M_0(X,\Bbb{Z}_4,\phi)$ is a sandwich near-ring. For $f \in M_0(X,\Bbb{Z}_4,\phi)$ we have $f(0)=0$, $f(1) \in \{0,1,2,3\}$ and $f(2) \in \{0,1,2,3\}$. Thus, $M_0(X,\Bbb{Z}_4,\phi)$ has $16$ elements. Let $f \in N_1$ and let $\psi_f : X \rightarrow \Bbb{Z}_4, x \mapsto f(x)$. As in the example before the function $h: N \rightarrow M_0(X,\Bbb{Z}_4,\phi), f \mapsto \psi_f$ is easily seen to be a near-ring isomorphism. 

Combinations of the concepts of centralizer near-rings and sandwich near-rings were used in \cite{Wendt3} and in \cite{FP} to describe $1$-primitive near-rings which do not necessarily have an identity. In \cite{Wendt3}, $1$-primitive near-rings are described as dense subnear-rings of near-rings of the type of $M_0(X,\Gamma,\phi,S):=\{f:X\longrightarrow \Gamma \mid f(0)=0\  \text{and}\
\forall s \in S\, \forall x \in X: f(s(x))=s(f(x))\}$, see Definition \ref{51D04}, where $S$ is an automorphism group of $\Gamma$ acting without fixed points on the set $X \setminus \{0\}$ and the near-ring operations are pointwise $+$ of functions and sandwich multiplication,  but the result and proof in \cite{Wendt3} requires that the primitive near-ring has a multiplicative right identity. We will see in this paper that this restriction is not needed. Thus we can generalize the result of \cite{Wendt3} to near-rings not necessarily having a multiplicative right identity and obtain a description of all zero symmetric $1$-primitive near-rings as well as $2$-primitive near-rings as dense subnear-rings of sandwich centralizer near-rings. This construction simplifies the construction of \cite{FP}. 

In \cite{FP} also no multiplicative right identity is required and the $1$-primitive near-ring is described as dense subnear-ring of a sandwich near-ring $M(X,N,\phi,\psi,B,C):=\{f:X\longrightarrow N \mid \forall s \in S\, \forall x \in X: f(s(x))=\psi(s)(f(x))\}$, where $X$ is a non-empty set, $(N,+)$ a group, $\phi$ the sandwich function, $B$ a subgroup of $\mathrm{Aut}(N,+)$, $S$ a group of permutations on $X$ and $\psi \in$ Hom($S,B$). This construction is more technical than that in \cite{Wendt3} and that we will use in our approach in this paper and the functions in $M(X,N,\phi,\psi,B,C)$ are not centralized by elements of $S$. For the details of the construction we refer the interested reader to \cite{FP}.

The idea of combining the concepts of centralizer near-rings and sandwich near-rings used in this paper allows us to explicitely describe and construct the sandwich function $\phi$ which determines the multiplication in the primitive near-ring. This will be done in the last section of this paper and will give us the possibility to construct zero symmetric $1$-primitive near-rings without necessarily having an identity systematically. Examples to demonstrate this construction are included.
\section{Sandwich centralizer near-rings}
The following definition introduces certain types of sandwich near-rings which were used by the author in \cite{Wendt2} and \cite{Wendt3} to describe near-rings with a multiplicative right identity.
\begin{defn}\label{51D04}
Let $(\Gamma,+)$ be a group, $X \subseteq \Gamma$ a subset of $\Gamma$ containing the zero $0$ of $(\Gamma,+)$  and $\phi:\Gamma \longrightarrow X$ a map such that $\phi(0)=0$. Let $S \subseteq \mathrm{End}(\Gamma,+)$, $S$ not empty, be such that $\forall s \in S, \forall \gamma \in \Gamma:\ \phi(s(\gamma)) = s(\phi(\gamma))$ and such that $S (X) \subseteq X$. Then $M_0(X,\Gamma, \phi,S):=\{f:X\longrightarrow \Gamma \mid f(0)=0\  \text{and}\ \forall s \in S, x \in X: f(s(x))=s(f(x))\}$ is a zero symmetric subnear-ring of $M_0(X,\Gamma, \phi)$ as defined in Definition \ref{Definition1}, which we call a sandwich centralizer near-ring.
\end{defn}

 It is straightforward to see that $M_0(X,\Gamma, \phi, S)$ is indeed a zero symmetric subnear-ring of $M_0(X,\Gamma, \phi)$ where the zero of $M_0(X,\Gamma,\phi, S)$ is the zero function $\overline 0$ on $X$. Since $S(X) \subseteq X$, $M_0(X,\Gamma, \phi, S)$ is not empty, since $\overline 0 $ is contained in $M_0(X,\Gamma, \phi, S)$.   

 Note that the function $id: X \rightarrow \Gamma, x \mapsto x$ is contained in $M_0(X, \Gamma, \phi, S)$ and serves as a multiplicative right identity of the near-ring. As shown in \cite{Wendt2}, any zero symmetric near-ring with a multiplicative right identity is isomorphic to a sandwich centralizer near-ring $M_0(X,\Gamma, \phi,S)$ with suitable $X, \Gamma, \phi, S$.
\section{The equivalence relation $\sim$}
Given a $1$-primitive near-ring and an $N$-group $\Gamma$ of type $1$ we now introduce an equivalence relation $\sim$ on $\Gamma$. What we need for the proof of our theorems in the next section is a special type of system of representatives w.r.t. $\sim$, being invariant under the $N$-automorphisms of $\Gamma$. The existence of such a representative system will be guaranteed in Lemma \ref{Lemma1}. 
\begin{defn}\label{Def1} Let $N$ be a near-ring and let $\Gamma$ be an $N$-group. Let $\gamma_1, \gamma_2 \in \Gamma$. Define $\gamma_1 \sim \gamma_2$ iff $\forall n \in N: n\gamma_1 = n\gamma_2$.
\end{defn}
It is easy to see that $\sim$ is an equivalence relation. To introduce another notation, we mention that when we have a function $f$ with domain $D$ and $M \subseteq D$, then $f_{|»M}$ means the restriction of the function to the set $M$.

In the proof of Lemma \ref{Lemma1} we will use that given an $N$-automorphism $s$ of an $N$-group $\Gamma$, then also the inverse function $s^{-1}$ is an $N$-automorphism of the $N$-group $\Gamma$. This is straightforward to see as is shown in the next proposition.
\begin{prop}\label{DCCN} Let $N$ be a zero symmetric near-ring and $\Gamma$ be an $N$-group. Let $s \in \mathrm{Aut}_N(\Gamma,+)$. Then also the inverse function $s^{-1} \in \mathrm{Aut}_N(\Gamma)$.
\end{prop}
\begin{proof} Let $s \in \mathrm{Aut}_N(\Gamma)$. Clearly, $s^{-1} \in \mathrm{Aut}(\Gamma,+)$. Let $n \in N$ and $\gamma \in \Gamma$. Then, $n\gamma=n(s(s^{-1}(\gamma)))=s(n(s^{-1}(\gamma)))$. Thus, $s^{-1}(n\gamma)=n(s^{-1}(\gamma))$. So we see that also $s^{-1}$ is an $N$-automorphism.
\end{proof}

\begin{lem}\label{Lemma1} Let $N$ be a zero symmetric near-ring and $\Gamma$ be an $N$-group of type $1$. Let $S:=\mathrm{Aut}_N(\Gamma,+)$. Then there is a set of representatives $X$ of the equivalence relation $\sim$ such that $S(X)\subseteq X$ and $0 \in X$.
\end{lem}
\begin{proof} Let  $\theta_1:=\{\gamma \in \Gamma | N\gamma = \Gamma\}$ be the set of generators and $\theta_0:=\{\gamma \in \Gamma | N\gamma = \{0\}\}$ be the set of non-generators of $\Gamma$. Since $\Gamma$ is an $N$-group of type $1$, $\Gamma = \theta_0 \cup \theta_1$. Let $D$ be a set of representatives w.r.t. $\sim$. We let $0$ be the representative for the equivalence class of $0$ and therefore, for any $\delta \in \theta_0$, $\delta \sim 0$. Hence, $D=X_1 \cup \{0\}$, with $X_1 \subseteq \theta_1$. Let $f$ be the function which maps every element in $\theta_1$ to its representative in $X_1$ w.r.t. $\sim$. Let $\gamma \in \theta_1$ and let $S(\gamma):=\{s(\gamma)|s \in S\}$ be the orbit of $\gamma$ under the action of $S$ on $\Gamma$. It is easy to see that $S(\gamma) \subseteq \theta_1$. Let $s_1, s_2 \in S$ and suppose that $s_1(\gamma) \sim s_2(\gamma)$. Then, for all $n \in N$, $s_1(n\gamma)=ns_1(\gamma)=ns_2(\gamma)=s_2(n\gamma)$. Since $\gamma \in \theta_1$ we see that for all $\delta \in \Gamma$, $s_1(\delta)=s_2(\delta)$ and therefore, $s_1=s_2$. This implies that the restriction $f_{|S(\gamma)}$ of $f$ to the orbit $S(\gamma)$ is an injective map. Let $K:=\{\cup_{\gamma \in J}S(\gamma)|J \subseteq \theta_1 \ \text{and}\  f_{|\cup_{\gamma \in J}S(\gamma)}\  \text{is injective}\}$. As we have just shown, $f$ is injective on any single orbit $S(\gamma)$, $\gamma \in \theta_1$. Consequently, $K$ is not the empty set. $K$ is ordered w.r.t. set inclusion $\subseteq$. Let $I$ be an index set such that for $i \in I$, $C_i \in K$ and $(C_i)_{i \in I}$ forms a chain in $K$. Let $M:=\cup_{i \in I}C_i$. So, $M=\cup_{i\in I}(\cup_{\gamma \in M_i}S(\gamma))$ where for $i \in I$, $M_i \subseteq \theta_1$ are suitable sets such that $(C_i)_{i \in I}$ forms a chain. So, the set $M$ is a union of  unions of orbits and consequently, $M$ is a union of orbits of elements from $\theta_1$. If we can show that $f$ is injective on $M$, then $M \in K$. Suppose $f$ is not injective on $M$. So there are $x,y \in M$, $x \neq y$ such that $f(x)=f(y)$. Thus, there are $j,l \in I$ such that $x \in C_j$ and $y \in C_l$. Since $(C_i)_{i \in I}$ forms a chain, we either have $C_j \subseteq C_l$ or $C_l \subseteq C_j$. So, either $x \in C_l$ and $y \in C_l$ or $x \in C_j$ and $y \in C_j$. $f$ is injective on $C_j$ as well as on $C_l$ and consequently we have $f(x) \neq f(y)$ which is a contradiction to the assumption that $f$ is not injective on $M$. Thus, $M \in K$ and $M$ is an upper bound for the chain $(C_i)_{i \in I}$. By Zorn's Lemma, $K$ contains a maximal element $R$, say. We claim that $R \cup \{0\}$ is a set of representatives w.r.t. $\sim$ which is invariant under the action of $N$-automorphisms of $\Gamma$. 

Note that as an element of $K$, $R$ is a union of orbits of $S$, so $S(R) \subseteq R$. Since $f$   is injective on $R$, any two elements of $R$ are in different equivalence classes. Suppose there is an element $\alpha \in \theta_1$ such that $\alpha \not \sim r$ for any $r \in R$. Let $s \in S$ and suppose there is an element $r \in R$ such that $s(\alpha) \sim r$. Since $s$ is an $N$-automorphism, also the inverse function $s^{-1}$ is an $N$-automorphism and contained in $S$ (see Proposition \ref{DCCN}). Thus, for all $n \in N$, $s(n\alpha)=ns(\alpha)=nr$ and therefore $n\alpha=ns^{-1}(r)$, so $\alpha \sim s^{-1}(r) \in R$. This is a contradiction to the assumption $\alpha \not \sim r$ for any $r \in R$ and so we see that for any $s \in S$ and any $r \in R$, $s(\alpha) \not \sim r$. But then, $f$ is injective on $R \cup S(\alpha)$, so $R \cup S(\alpha) \in K$. Clearly, $S(\alpha) \not \subseteq R$ and so $R$ is properly contained in $R \cup S(\alpha)$. This contradicts the maximality of $R$. Since we also  have $\delta \sim 0$ for any element $\delta \in \theta_0$ we see that $X:=R \cup \{0\}$ is a set of representatives w.r.t $\sim$. Since $S(0)=0$ and $S(R) \subseteq R$, $S(X) \subseteq X$.
\end{proof}
We keep the notation of Lemma \ref{Lemma1} to give some comments. We have seen in the proof of the lemma, that given an element $\gamma \in \theta_1$, then $S(\gamma)$ is a set of $\sim$ inequivalent elements. It is easy to see that $S(\theta_1) \subseteq \theta_1$. Consequently, the result of Lemma \ref{Lemma1} is immediate and we would not have to apply Zorn's Lemma if the $N$-group $\Gamma$ has finitely many orbits w.r.t. the action of $S$ on $\Gamma$, in particular  this is the case when $\Gamma$ is finite.

\section{Density theorems}
We will now prove two theorems which are our main theorems of this paper. First we show that up to isomorphism zero symmetric and $1$-primitive near-rings show up as dense subnear-rings of sandwich centralizer  near-rings with special conditions on $X, \Gamma, \phi$ and $S$. Following this theorem we then can easily prove a similar result for $2$-primitive near-rings. We should make clear what density means (see also \cite{Pilz1}, Proposition 4.26) and fix some more notation.

\begin{defn}\label{Dens} $F$ is a dense subnear-ring of $M_0(X, \Gamma, \phi, S)$ if and only if $\forall s \in \Bbb{N}$ $\forall x_1, \ldots , x_s \in X$ $\forall g \in M_0(X, \Gamma, \phi, S)$$\exists f \in F$: $f(x_i)=g(x_i)$ for all $i \in \{1, \ldots, s\}$. 
\end{defn}
Also, we need the concept of fixedpointfreeness. 
\begin{defn} Let $S$ be a group of automorphisms of a group $\Gamma$. Let $M \subseteq \Gamma\setminus \{0\}$ such that $S(M) \subseteq M$. $S$ is called fixedpointfree on $M$ if for $s \in S$ and $m \in M$, $s(m)=m$ implies $s = id$, $id$ being the identity function. $S$ is called a fixedpointfree automorphism group of $\Gamma$ if it acts fixedpointfree on $\Gamma \setminus \{0\}$.
\end{defn}
Let $(\Gamma, +)$ be a group. If $I$ is a normal subgroup of $\Gamma$ we will denote this as $I \triangleleft \Gamma$. For $\delta \in \Gamma$, $\delta+I$ is the coset of $\delta$. $\emptyset$ stands for the empty set. 

As already pointed out in the introduction, $1$-primitive near-rings which are rings are primitive rings in the ring theoretical sense (see \cite{Pilz1}, Proposition 4.8). So, we restrict our discussion to non-rings. We are now ready to formulate our first  theorem.

\begin{thm}\label{EEst1}  Let $N$ be a zero symmetric near-ring which is not a ring. Then the following are equivalent:
\begin{enumerate}
\item $N$ is $1$-primitive.
\item There exist 
\begin{enumerate}
\item a group $(\Gamma,+)$,
\item  a set $X=\{0\} \cup X_1 \subseteq \Gamma$, $X_1 \neq \emptyset$, $0 \not \in X_1$ and $0$ being the zero of $\Gamma$,
 \item $S \leq \mathrm{Aut}(\Gamma,+)$, with $S(X) \subseteq X$ and $S$ acting without fixed points on $X_1$,
\item a function $\phi : \Gamma \rightarrow X$ with $\phi_{|X} = id$, $\phi(0)=0$ and such that $\forall \gamma \in \Gamma\ \forall s \in S:\phi(s(\gamma)) = s(\phi (\gamma))$,

\end{enumerate}

such that $N$ is isomorphic to a dense subnear-ring $M_S$ of $M_0(X, \Gamma, \phi, S)$ where $X, \Gamma, \phi, S$ additionally satisfy the following property (P):

Let $\Gamma_0:=\{\gamma \in \Gamma | \phi(\gamma)=0\}$ and $C:=\{I \triangleleft \Gamma | I \subseteq \Gamma_0\ \text{and}\ \Gamma_0=\cup_{\delta \in \Gamma_0}\delta+I\ \text{and}\ \forall \gamma \in \Gamma\setminus \Gamma_0 \forall i \in I: S(\phi(\gamma + i)) = S(\phi(\gamma))\}$. Then $I \in C \Rightarrow (I=\{0\}\ \text{or}\ \exists i \in I \exists \gamma_1 \in \Gamma \setminus \Gamma_0 \exists s \in S \exists \gamma \in \Gamma: \phi(\gamma_1+i)=s(\phi(\gamma_1))\ \text{and}\ s(\gamma)-\gamma \not \in I$).

\end{enumerate}
\end{thm}

\begin{proof} $(1) \Rightarrow (2)$: Let $N$ act $1$-primitively on the $N$-group $\Gamma$. Let  $\theta_1:=\{\gamma \in \Gamma | N\gamma = \Gamma\}$ be the set of generators and $\theta_0:=\{\gamma \in \Gamma | N\gamma = \{0\}\}$ be the set of non-generators of $\Gamma$. By $1$-primitivity of $N$, $\Gamma = \theta_1 \cup \theta_0$. Let $S=\mathrm{Aut}_N(\Gamma,+)$. On $\Gamma$ we define the equivalence relation $\sim$ as in Definition \ref{Def1} by $\gamma_1 \sim \gamma_2$ iff for all $n \in N$, $n\gamma_1 = n\gamma_2$. According to Lemma \ref{Lemma1} we choose a set of representatives $X$ of the equivalence relation $\sim$ in a way that $S(X) \subseteq X$ and $0$ is the representative of the equivalence class of $0$. Note that any element in $\theta_0$ is equivalent to $0$ w.r.t. $\sim$. Thus, $X=X_1\cup \{0\}$, where $X_1 \subseteq \theta_1$. Note that $X_1 \neq \emptyset$ because as an $N$-group of type $1$, $\Gamma$ has a generator.

Let $\phi: \Gamma \rightarrow X, \gamma \mapsto x$ where $\gamma \sim x$. Then $\phi_{|X}=id$ and $\phi (0)=0$ because the representative of the zero equivalence class was taken to be zero.

Next we show that for all $\gamma \in \Gamma$ and $s \in \mathrm{Aut}_N(\Gamma,+)$ we have $s(\phi(\gamma))=\phi(s(\gamma))$. To show this, we first prove that $s(\phi(\gamma)) \sim s(\gamma)$ for any $\gamma \in \Gamma$. Let $n \in N$. Then, $ns(\phi(\gamma))=s(n\phi(\gamma))$. Now, $\phi(\gamma) \sim \gamma$, so $s(n\phi(\gamma))=s(n\gamma)=ns(\gamma)$. This shows that $s(\phi(\gamma)) \sim s(\gamma)$. Consequently, by the definition of $\phi$, $\phi(s(\phi(\gamma)))=\phi(s(\gamma))$. Since $S(X) \subseteq X$ and $\phi_{|X}=id$ we have that $\phi(s(\phi(\gamma)))=s(\phi(\gamma))$ which proves the desired property.

Next we prove fixedpointfreeness of $S$ on $X_1$. Note that $S(X_1) \subseteq X_1$ because $S(X) \subseteq X$ and $S$ is a group of automorphisms, so only the zero $0$ in $X$ is mapped to zero. Let $\gamma \in X_1 \subseteq \theta_1$. Then $N\gamma = \Gamma$. Suppose that $s(\gamma)=\gamma$. Therefore, for all $n \in N$ we get $ns(\gamma)=s(n\gamma)=n\gamma$. Thus, for all $\delta \in \Gamma$ we have $s(\delta)=\delta$ and $s=id$.

 Now we show how to embed $N$ into $M_0(X, \Gamma, \phi, S)$. For every $n \in N$, let $f_n$ be the function $f_n:X \longrightarrow \Gamma, x \mapsto nx$. We now prove that the mapping $h: n \mapsto f_n$ is an embedding of $N$ into $M_0(X, \Gamma, \phi, S)$.

First we show that for $n \in N$, $f_n \in M_0(X,\Gamma, \phi, S)$, so for all $s \in S$, for all $x \in X$ we must have $s(f_n(x))=f_n(s(x))$. Since $s$ is an $N$-automorphism we get $f_n(s(x))= n(s(x))= s(nx)=s(f_n(x))$ and consequently, $f_n \in M_0(X, \Gamma, \phi, S)$ since also $f_n(0)=0$. So, $h$ maps $N$ into $M_0(X, \Gamma, \phi, S)$.

$h$ is a near-ring homomorphism: Let $j$ and $k$ be arbitrary elements of $N$. Then $h(j+k)=f_{(j+k)}= f_j +f_k$ by right distributivity of $N$. Let $x \in X$. Then $h(jk)(x)=f_{jk}(x)=(jk)x$. On the other hand, $h(j) \circ'h(k) =  f_j \circ \phi \circ  f_k$. So, for every $x \in X$, $f_j \circ \phi \circ f_k(x)= j(\phi(kx))$. By definition of $\phi$ we know that $\phi(kx) \sim kx$ and consequently, $j(\phi(kx))=j(kx)=(jk)x$. This shows that $h(j)\circ'h(k)=f_{jk}=h(jk)$.

$h$ is injective: Since $h$ is a near-ring homomorphism, it suffices to show that the kernel of $h$ is zero. Suppose there exists an element $j \in N$ such that  $f_{j}$ is the zero function. This means that $f_j(x)=0$ for all $x \in X$. Since $X$ is a set of representatives w.r.t. $\sim$, this implies $j\Gamma = \{0\}$. By faithfulness of $\Gamma$ we get $j=0$. Hence, $h$ is injective and this finally proves that $h$ is an embedding.

So we can embed $N$ into the near-ring $M_0(X, \Gamma, \phi, S)$ and we let $h(N)=:M_S$. Consequently, $N \cong M_S$ and it remains to show that $M_S$ is a dense subnear-ring of $M_0(X,\Gamma, \phi, S)$  where $X, \Gamma, \phi, S$ additionally satisfy the property (P).

Suppose  $x_1 \in X_1$ and $x_2 \in X_1$ are from different orbits of $S$ acting on $X_1$ and suppose $x_1$ and $x_2$ have the same annihilator. Then, $Nx_1=Nx_2=\Gamma$ and $s: \Gamma \longrightarrow \Gamma, nx_1 \mapsto nx_2$ is a well defined $N$-automorphism of $\Gamma$, which is straightforward to see. Since $x_1 \in \theta_1$, there is an element $k \in N$ such that $kx_1=x_1$. Consequently, $s(x_1)=kx_2$. For any $n \in N$ we have $nx_1=nkx_1$ and therefore $n-nk \in (0:x_1)=(0:x_2)$. It follows that $nx_2=nkx_2$ for all $n \in N$ which means that $s(x_1)=kx_2 \sim x_2$. Since $s(x_1) \in X_1$, $s(x_1)=\phi(s (x_1))=x_2$ which contradicts the assumption that $x_1 \in X_1$ and $x_2 \in X_1$ are from different orbits of $S$ acting on $X_1$. So, elements from different orbits of $S$ acting on $X_1$ must have different annihilators.

Take $x_1,\ldots, x_l \in X_1$, $l \in \Bbb{N}$, from finitely many  different orbits of $S$ acting on  $X_1$. Then all elements in $\{x_1, \ldots, x_l\}$ have different annihilators as we have just shown. Since $N$ is not a ring we can apply Theorem 4.30 of \cite{Pilz1} to get that for all $\gamma_1, \ldots, \gamma_l \in \Gamma$ there exists some $n \in N$ such that $nx_i=\gamma_i$ for all $i \in \{1, \ldots, l\}$. Hence, for all $i \in \{1,\ldots, l\}$, $h(n)(x_i)=f_n(x_i)=nx_i=\gamma_i$. 

 We now have to show that $\forall l \in \Bbb{N}$ $\forall x_1, \ldots , x_l \in X$ $\forall f \in M_0(X, \Gamma, \phi, S)$$\exists m \in M_S$: $m(x_i)=f(x_i)$ for all $i \in \{1, \ldots, l\}$. 

For any function $f \in M_0(X, \Gamma, \phi, S)$ and any function $m \in M_S$ we have $f(0)=m(0)=0$.  So it suffices to consider the case that for $l \in \Bbb{N}$, $x_1, \ldots , x_l \in X_1$. Let $f \in M_0(X, \Gamma, \phi, S)$. Let $v \in \Bbb{N}$ and let $z_1, \ldots , z_v$ be a set of orbit representatives for the elements $x_1, \ldots, x_l \in X_1$ under the action of $S$ on $X_1$. Thus, $z_1, \ldots , z_v$ have different annihilators and so there is an element $m \in N$ such that $f_m(z_i)=f(z_i)$ for $i \in \{1, \ldots, v\}$. For $k \in \{1, \ldots, l\}$ there exists a unique $j \in \{1, \ldots ,v\}$ such that $x_k \in S(z_j)$. Then $x_k=s(z_j)$ for some, by fixedpointfreeness of $S$ on $X_1$, unique $s$ and so, $f_m(x_k)=f_m(s(z_j))=s(f_m(z_j))=s(f(z_j))=f(s(z_j))=f(x_k)$. So, $f_m$ and $f$ are equal functions when restricted to the set $\{x_1, \ldots , x_l\}$ and also $f_m \in M_S$. This proves density of $M_S$ in $M_0(X, \Gamma, \phi, S)$.

Finally we have to show that $X, \Gamma, \phi, S$ satisfy the property (P). Let $I \in C$. Assume that $I \neq \{0\}$. $I \neq \Gamma$ since $\Gamma_0 \neq \Gamma$, so $I$ is a non-trivial normal subgroup of $\Gamma$. Suppose property (P) does not hold, so assume that  for all $i \in I$, for all $\gamma_1 \in \Gamma \setminus \Gamma_0$, for all $s \in S$ and for all $\gamma \in \Gamma $ either $\phi(\gamma_1+i) \neq s(\phi(\gamma_1))$ holds or $s(\gamma)-\gamma \in I$ holds. Since $I \in C$, for all $i \in I$ and for all $\gamma_1 \in \Gamma\setminus \Gamma_0$ there is an $s \in S$ such that $\phi(\gamma_1+i)=s(\phi(\gamma_1))$ and therefore, for this $s \in S$, $s(\gamma)-\gamma \in I$ for all $\gamma \in \Gamma$.

 Let $n \in N$ and $\gamma_0 \in \Gamma_0$. Thus, $\gamma_0+I \subseteq \Gamma_0$ because by definition of the elements in $C$, $\Gamma_0$ is a union of cosets of $I$. Consequently, for all $j \in I$, $\gamma_0+j \in \gamma_0 + I \subseteq \Gamma_0$. Thus, $n(\gamma_0+j) - n\gamma_0 = n\phi(\gamma_0+j)-n\phi(\gamma_0)=0-0 \in I$. Let $\gamma_1 \in \Gamma \setminus \Gamma_0$ and $i \in I$. Consequently, $\phi(\gamma_1+i)=s(\phi(\gamma_1))$ for some $s \in S$ and so, since property (P) is assumed not to hold, $s(\gamma)-\gamma \in I$ for all $\gamma \in \Gamma$. So, for all $n \in N$, $n(\gamma_1+i)-n\gamma_1=n\phi(\gamma_1+i)-n\phi(\gamma_1)= ns(\phi(\gamma_1))-n\phi(\gamma_1)=s(n\phi(\gamma_1))-n\phi(\gamma_1) \in I$.  This shows that $I$ is a non-trivial and proper $N$-ideal of $\Gamma$, contradicting that $N$ is 1-primitive on $\Gamma$. Hence, property (P) must hold.

$(2) \Rightarrow (1)$:  We have to show that $N \cong M_S$ is a $1$-primitive near-ring. $\Gamma$ is an $M_S$-group in a natural way by defining the action $\odot$ of $M_S$ on $\Gamma$ as $m \odot \gamma:= m(\phi (\gamma))$ for $m \in M_S$ and $\gamma \in \Gamma$. Since $\phi_{|X}=id$ we have $X=\phi (\Gamma)$, so $\Gamma$ is a faithful $M_S$-group.

Let $\gamma \in \Gamma$ and suppose $\phi(\gamma)=0$, so $\gamma \in \Gamma_0$. Then clearly $M_S\odot \gamma = \{0\}$. On the other hand there exist elements $\gamma \in \Gamma$, such that $0 \neq \phi(\gamma) \in X_1$ because $\phi_{|X}=id$ and $X_1 \neq \emptyset$. Let $\gamma$ be such that $\phi(\gamma) \in X_1$, so $\gamma \in \Gamma \setminus \Gamma_0$. Let $\delta \in \Gamma$. We now define a function $f: X \longrightarrow \Gamma$ in the following way: $f(\phi(\gamma)):=\delta$. Let $s \in S$. Since $S$ acts without fixedpoints on $X_1$, we can well define $f(s(\phi(\gamma))):=s(f(\phi(\gamma)))=s(\delta)$ and $f(X \setminus S(\phi(\gamma))):=\{0\}$. From the definition of $f$ we see that $f \in M_0(X,\Gamma, \phi, S)$, so by density of $M_S$, there is a function $m \in M_S$ with $m(\phi(\gamma))=f(\phi(\gamma))=\delta$. Since $\delta \in \Gamma$ was chosen arbitrary, this shows that $M_S \odot \gamma = \Gamma$. Consequently, for $\delta \in \Gamma_0$ we have $M_S\odot \delta = \{0\}$ and for $\gamma \in \Gamma \setminus \Gamma_0$ we have $M_S\odot \gamma = \Gamma$. 

We now show that there are no non-trivial $M_S$-ideals in $\Gamma$. Suppose that $I$ is a non-trivial $M_S$-ideal of $\Gamma$. Then $I$ is a non-trivial normal subgroup of $(\Gamma,+)$, and for all $m \in M_S$, $\gamma \in \Gamma$ and $i \in I$ we have that $m \odot (\gamma+i)-m\odot \gamma=m(\phi(\gamma+i))-m(\phi(\gamma)) \in I$. Since $m(\phi(0))=0$ for all $m \in M_S$, $I$ being an $M_S$-ideal implies that $m(\phi(i)) \in I$ for all $i \in I$ and all $m \in M_S$. This implies that $I \subseteq \Gamma_0$. Let $\delta \in \Gamma_0$, $i \in I$. Then, for all $m \in M_S$, $m(\phi(\delta+i))-m(\phi(\delta))=m(\phi(\delta+i)) \in I$. This shows that $\delta+i \in \Gamma_0$ and so, $\Gamma_0$ is a union of cosets of $I$. 

Assume that $I$ is not contained in the set $C$. Thus, there exists an element $\gamma \in \Gamma \setminus \Gamma_0$ and an element $i \in I$ such that $S(\phi(\gamma + i)) \neq S(\phi(\gamma))$. Since $S$ is a group, this implies that for all $s \in S$, $\phi(\gamma+i) \neq s(\phi(\gamma))$. This means that $\phi(\gamma)$ and $\phi(\gamma + i)$ are in different orbits of $S$ acting on  $X$.  Suppose that $\gamma + i \in \Gamma_0 $. Since $I$ is an $M_S$-ideal, this implies that $m(\phi(\gamma)) \in I$ for all $m \in M_S$, hence we must have $\phi(\gamma)=0$, a contradiction to $\gamma \in \Gamma\setminus \Gamma_0$. So we have that $\phi(\gamma + i) \neq 0$ as well as $\phi(\gamma) \neq 0$ and we now define two functions $f_1: X \longrightarrow \Gamma$ and $f_2: X \longrightarrow \Gamma$. 

Let $f_1(\phi(\gamma)):=\delta_1 \in I$, for $s \in S$ let $f_1(s(\phi(\gamma))):=s(f_1(\phi(\gamma)))$ and $f_1(X\setminus S(\phi(\gamma))):=\{0\}$. Let $f_2(\phi(\gamma +i)):=\delta_2 \not \in I$, for $s \in S$ let $f_2(s(\phi(\gamma + i))):=s(f_2(\phi(\gamma + i)))$ and $f_2(X\setminus S(\phi(\gamma+i))):=\{0\}$. $f_1$ and $f_2$ are well defined because of fixedpointfreeness of $S$ on $X_1$. It is a routine check to see that $f_1$ and $f_2$ are elements in $M_0(X, \Gamma, \phi, S)$. Since for all $s \in S$, $\phi(\gamma+i) \neq s(\phi(\gamma))$ we have that $f_1(\phi(\gamma+i))=0$ as well as $f_2(\phi(\gamma))=0$.

We now have that $(f_1+f_2)(\phi(\gamma + i))-(f_1+f_2)(\phi (\gamma))$=$f_1(\phi(\gamma+i))+f_2(\phi(\gamma+i))-f_2(\phi(\gamma))-f_1(\phi(\gamma))$=$0+\delta_2-0-\delta_1 \not \in I$. By density of $M_S$ in $M_0(X,\Gamma, \phi, S)$, there is an element $m \in M_S$ such that $m(\phi(\gamma + i))=(f_1+f_2)(\phi(\gamma + i))$ and $m(\phi (\gamma))=(f_1+f_2)(\phi (\gamma))$. Consequently, $I$ is not an $M_S$-ideal.

Assuming that $I$ is not contained in the set $C$ contradicts our assumption that $I$ is an $M_S$-ideal. So, we now assume that $I$ is contained in $C$. 
Consequently, by property (P), there exists $i \in I$, there exists $\gamma_1 \in \Gamma \setminus \Gamma_0$,  there exists $s \in S$ and there exists $\gamma \in \Gamma$ such that $\phi(\gamma_1+i)=s(\phi(\gamma_1))$ and $s(\gamma)-\gamma \not \in I$. Since $\gamma_1 \in \Gamma \setminus \Gamma_0$, $M_S \odot \gamma_1 = \Gamma$, so there exists an element $m \in M_S$ such that $m(\phi(\gamma_1))=\gamma$. Consequently, $m (\phi(\gamma_1+i))-m (\phi(\gamma_1))=m(s(\phi(\gamma_1)))-m(\phi(\gamma_1))=s(m(\phi(\gamma_1)))-m(\phi(\gamma_1))=s(\gamma)-\gamma \not \in I$. This is again a contradiction to the assumption that $I$ is an $M_S$-ideal. 

This shows that there exist no non-trivial $M_S$-ideals in $\Gamma$, so $M_S$ is 1-primitive on $\Gamma$.

\end{proof}
Property (P) of Theorem \ref{EEst1} is a technical condition which excludes subgroups of $\Gamma$ to be $M_S$-ideals, in the language of Theorem \ref{EEst1}. Since any $2$-primitive near-ring is also $1$-primitive, Theorem \ref{EEst1} applies to zero symmetric and $2$-primitive near-rings also. When considering $2$-primitive near-rings the technical condition of property (P) can be much more simplified. This leads to an especially simple version of the theorem.

\begin{thm}\label{EEst2}  Let $N$ be a zero symmetric near-ring which is not a ring. Then the following are equivalent:
\begin{enumerate}
\item $N$ is $2$-primitive.
\item There exist 
\begin{enumerate}
\item a group $(\Gamma,+)$,
\item  a set $X=\{0\} \cup X_1 \subseteq \Gamma$, $X_1 \neq \emptyset$, $0 \not \in X_1$ and $0$ being the zero of $\Gamma$,
 \item $S \leq \mathrm{Aut}(\Gamma,+)$, with $S(X) \subseteq X$ and $S$ acting without fixed points on $X_1$,
\item a function $\phi : \Gamma \rightarrow X$ with $\phi_{|X} = id$, $\phi(0)=0$ and such that $\forall \gamma \in \Gamma\ \forall s \in S:\phi(s(\gamma)) = s(\phi (\gamma))$,
\end{enumerate}

such that $N$ is isomorphic to a dense subnear-ring $M_S$ of $M_0(X, \Gamma, \phi, S)$ where  $\Gamma_0:=\{\gamma \in \Gamma | \phi(\gamma)=0\}$ does not contain any non-trivial subgroups of $\Gamma$.

\end{enumerate}
\end{thm}
\begin{proof}  $(1) \Rightarrow (2):$ Since $N$ is $2$-primitive it is also $1$-primitive and hence Theorem \ref{EEst1} and its proof of $(1) \Rightarrow (2)$ applies. So, let $\Gamma$ be the $N$-group of type $2$ the near-ring acts on $2$-primitively. Let $\phi$ be as in the proof of $(1) \Rightarrow (2)$ of Theorem \ref{EEst1}. It only remains to show that $\Gamma_0$ does not contain any non-trivial subgroups of $\Gamma$. Suppose $K \subseteq \Gamma_0$ is a subgroup of $\Gamma$. Since $\phi(K)=\{0\}$ we know from the definition of $\phi$ that $K\subseteq \theta_0=\{\gamma \in \Gamma | N\gamma = \{0\}\}$. Thus, $NK = \{0\} \subseteq K$ and $K$ is an $N$-subgroup of $\Gamma$. It follows from $2$-primitivity of $N$ that $K=\{0\}$.

$(2) \Rightarrow (1):$ As in the proof of $(2) \Rightarrow (1)$ of Theorem \ref{EEst1} one shows that $M_S$ acts faithfully on $\Gamma$ with the action $\odot$ and such that for $\delta \in \Gamma_0$ we have $M_S\odot \delta = \{0\}$ and for $\gamma \in \Gamma \setminus \Gamma_0$ we have $M_S\odot \gamma = \Gamma$. Also, $\Gamma \setminus \Gamma_0 \neq \emptyset$. Suppose that $K$ is an $M_S$-subgroup of $\Gamma$ and $K \neq \Gamma$. It follows  that $K \subseteq \Gamma_0$. By assumption, $\Gamma_0$ does not contain any non-trivial subgroups of $\Gamma$. Hence, $K=\{0\}$. Thus, $\Gamma$ contains no non-trivial $M_S$-subgroups and $M_S$ is $2$-primitive on $\Gamma$.

\end{proof}

\section{Construction of $\phi$ and examples}
 We keep the notation of the proof of Theorem \ref{EEst1} throughout the whole section. In order to construct $1$-primitive near-rings as sandwich centralizer near-rings with the help of Theorem \ref{EEst1} one must be assured that $X, \Gamma, \phi, S$ satisfy property (P). In case $X=\Gamma$ and $\phi=id$ we have that $\Gamma_0=\{0\}$. In this case property (P) is trivially fulfilled because only the trivial group $\{0\}$ is contained in $C$. In fact, in this case $M_0(X,\Gamma,\phi,S)= M_S(\Gamma):=\{f: \Gamma \rightarrow \Gamma | \forall \gamma \in \Gamma \forall s \in S: s(f(\gamma))=f(s(\gamma))\ \text{and}\ f(0)=0\}$. So, $M_0(X,\Gamma,\phi,S)= M_S(\Gamma)$ is a $1$-primitive near-ring with identity element (and thus $2$-primitive). The fact that all zero symmetric $1$-primitive near-rings with identity element which are not rings show up as dense subnear-rings of near-rings of the type $M_S(\Gamma)$ with $S$ a group of fixedpointfree automorphisms acting on $\Gamma$ is certainly the most well known density theorem for primitive near-rings (see \cite{Pilz1}, Theorem $4.52$). This result is also covered by Theorem \ref{EEst1}, Theorem \ref{EEst2} respectively, because in case of a near-ring with identity, $\phi$ as constructed in the proof of $(1) \Rightarrow (2)$ of Theorem \ref{EEst1} is just the identity function. So, $X= \Gamma$, $\Gamma_0=\{0\}$ and $M_0(X, \Gamma, \phi, S)=M_S(\Gamma)$. $S$ is acting  without fixed points on $\Gamma \setminus \{0\}$ because of condition (2c) in Theorem \ref{EEst1}.

 We need not only consider near-rings with identity to obtain situations when property (P) is easily fulfilled. Property (P) is obviously fulfilled when $C$ only contains the trivial group $\{0\}$. $C$ will only contain the trivial group  for example when $\Gamma_0$ is not a union of cosets of some non-trivial normal subgroup of $\Gamma$. Hence, probably the easiest way to obtain $1$-primitive near-rings is to take $S=\{id\}$ and any function $\phi$ mapping from $\Gamma$ to a subset $X$ of $\Gamma$ containing the zero $0$ such that $\Gamma_0$ is not a union of cosets of some non-trivial normal subgroup of $\Gamma$. Note that the examples after Definition \ref{Definition1} are of that type.

To be in a position to construct $1$-primitive near-rings or $2$-primitive near-rings with the help of our main theorems and $S \neq \{id\}$ one has to find a suitable sandwich function $\phi$ which commutes with the automorphisms in $S$ and $S$ has to act without fixed points on the set $X_1=X \setminus \{0\}$. If one has found such a function $\phi$ and the group $S$, then primitivity of the near-ring $M_0(X, \Gamma, \phi, S)$ only depends on the subgroups contained in $\Gamma_0$. In the following two propositions we show how to construct the sandwich function $\phi$ with the desired properties and show that any such $\phi$ can be constructed this way. 
\begin{prop}\label{Gabi1} Let $(\Gamma,+)$ be a group and $S \leq \mathrm{Aut}(\Gamma,+)$. Let $G \subseteq \Gamma \setminus \{0\}$ such that $S(G) \subseteq G$ and $S$ acts without fixed points on $G$. Let $\{e_i | i \in I\}$, $I$ a suitable index set, be a complete set of orbit representatives of the orbits of $S$ acting on $G$. So, $G=\cup_{i \in I}S(e_i)$. Let $\emptyset \neq J$,  $J \subseteq I$.  Let $X_1:=\cup_{j \in J} S(e_j)$ and $X:= \{0\} \cup X_1$. Let $K:=I\setminus J$. If $\emptyset \neq K$, let $f: \{e_k | k \in K\} \longrightarrow \cup_{j \in J} S(e_j)$ be a function. Define $\phi: \Gamma \longrightarrow X$ as \[\phi(\gamma):=\left\{\begin{array}{lll} 0 & \mbox{if $\gamma \in \Gamma \setminus G$} \\ \gamma & \mbox{if $\gamma \in \cup_{j \in J}S(e_j)$} \\ s(f(e_k)) & \mbox{if $K$ is not empty and $\gamma =s(e_k) \in \cup_{k \in K}S(e_k)$} \end{array} \right. \] Then, $\phi$ is a well defined function such that $\phi_{|X}=id$ and $\forall \gamma \in \Gamma\ \forall s \in S:\phi(s(\gamma)) = s(\phi (\gamma))$. Furthermore, $S$ acts without fixed points on $X_1$ and $S(X_1) \subseteq X_1$.
\end{prop}
\begin{proof}
Since  $X_1 \subseteq G$, $S$ acts without fixed points on $X_1$. Also, since $X_1$ is a union of orbits of $S$, $S(X_1) \subseteq X_1$. Suppose $K:=I \setminus J$ is not the empty set. Let $f:\{e_k | k \in K\} \longrightarrow \cup_{j \in J}S(e_j)$ be a function. Let $\gamma  \in \cup_{k \in K}S(e_k)$. By fixedpointfreeness of $S$ on $G$ there is a unique $s \in S$ and a unique $k \in K$ such that $\gamma = s(e_k)$. Then the definition $\phi(\gamma):=s(f(e_k))$ makes $\phi$  a well defined function. Since $0 \in \Gamma \setminus \cup_{i \in I}S(e_i)$ we have $\phi(0)=0$ and so, $\phi_{|X}=id$. It remains to show that for all $\gamma \in \Gamma$ and all $s \in S$, $s(\phi(\gamma))=\phi(s(\gamma))$. Let $\gamma \in \Gamma \setminus \cup_{i \in I}S(e_i)$ and $s \in S$. Then, $s(\phi(\gamma))=s(0)=0$ and since $s(\gamma) \in \Gamma \setminus \cup_{i \in I}S(e_i)$ we also have $\phi(s(\gamma))=0$. Let $\gamma \in \cup_{j \in J}S(e_j)$ and $s \in S$. Then, $s(\phi(\gamma))=s(\gamma)$. On the other hand, $s(\gamma) \in \cup_{j \in J}S(e_j)$ and so we also have $\phi(s(\gamma))=s(\gamma)$ by definition of $\phi$. Suppose $K$ is not empty. Let $\gamma \in \cup_{k \in K}S(e_k)$. So, $\gamma = s_1(e_k)$ for a unique $s_1 \in S$. Let $s \in S$. Then, $s(s_1(e_k)) \in \cup_{k \in K}S(e_k)$. Then, $\phi(s(\gamma))=\phi(s(s_1(e_k)))=s(s_1(f(e_k)))$. On the other hand, we also have $s(\phi(\gamma))=s(\phi(s_1(e_k)))=s(s_1(f(e_k)))$. This finally shows that $\phi(s(\gamma))=s(\phi(\gamma))$ for all $\gamma \in \Gamma$.
\end{proof}
We give a concrete example to demonstrate the construction process of the function $\phi$ in Proposition \ref{Gabi1}. We use the notation of Proposition \ref{Gabi1}. Let $(\Gamma,+):=(\Bbb{Z}_7,+)$ and $S:=\{id, -id\}$ where $-id : \Bbb{Z}_7 \rightarrow \Bbb{Z}_7, x \mapsto -x$. $S$ is a fixedpointfree automorphism group of $(\Bbb{Z}_7,+)$. Let $G:=\{1,6,2,5\}$ be the union of the orbits of $6$ and $5$. Let $e_1:=6$ and $e_2:=5$, so $I:=\{1,2\}$. Let $J:=\{1\}$ and $K:=\{2\}$. Thus, $X_1:=\{1,6\}$, the orbit of $6$, and $X:=\{0,1,6\}$. Let $f:\{5\} \rightarrow \{1,6\}, f(5)=1$ (here we could also define $f(5)=6$ resulting in a different $\phi$). Then $\phi: \Bbb{Z}_7 \rightarrow X$ is defined as follows: $0=\phi(0)=\phi(3)=\phi(4)$, $\phi(1)=1$, $\phi(6)=6$. Since $2=-id(e_2)=-id(5)$ and $5=id(e_2)=id(5)$ we get $\phi(2)=-id(f(5))=-id(1)=6$ and $\phi(5)=id(f(5))=1$. Clearly, $\phi_{|X}=id$ and for all $\gamma \in \Bbb{Z}_7$ we have $\phi(-\gamma)=-\phi(\gamma)$ as is easily seen. Thus, $\phi$ has all the desired properties as claimed in Proposition \ref{Gabi1}. The sandwich centralizer near-ring $M_0(X, \Gamma, \phi, S)$ constructed using these groups $S$ and $\Gamma$ and this set $X$ and function $\phi$ fulfilles Theorem \ref{EEst2} because $(\Gamma,+)$ is a simple group. Thus, $M_0(X, \Gamma, \phi, S)$ is a $2$-primitive near-ring without an identity element.

The next proposition shows that any sandwich function of the type we require in the Theorems \ref{EEst1} and \ref{EEst2} is of the form as constructed in Proposition \ref{Gabi1}.
\begin{prop} Let $(\Gamma,+)$ be a group and $S \leq \mathrm{Aut}(\Gamma,+)$. Let $\emptyset \neq X_1 \subseteq \Gamma \setminus \{0\}$ such that $S(X_1) \subseteq X_1$ and $S$ acts without fixed points on $X_1$. Let $X:=\{0\} \cup X_1$. Let $\phi: \Gamma \longrightarrow X$ be a function such that $\phi_{|X}=id$ and such that $\forall \gamma \in \Gamma\ \forall s \in S:\phi(s(\gamma)) = s(\phi (\gamma))$. Let  $\Gamma_0:=\{\gamma \in \Gamma | \phi(\gamma)=0\}$. Then the following hold:
\begin{enumerate}
\item $G:=\Gamma \setminus \Gamma_0$ is a non-empty set such that $S(G) \subseteq G$ and $S$ acts without fixed points on $G$. Thus, $\phi(\gamma)=0$ if $\gamma \in \Gamma \setminus G$.
\item Let $\{e_i | i \in I\}$, $I$ a suitable index set, be a complete set of orbit representatives of the orbits of $S$ acting on $G$. Then there is a non-empty subset $J \subseteq I$ such that  $\{e_j | j \in J\}$ is a complete set of orbit representatives of the orbits of $S$ acting on $X_1$. Thus, $X=\{0\} \cup_{j \in J} S(e_j)$ and $\phi(\gamma)=\gamma$ if $\gamma \in \cup_{j \in J} S(e_j)$.
\item If $K:=I\setminus J$ is not the empty set, then there is a function $f: \{e_k | k \in K\} \longrightarrow \cup_{j \in J} S(e_j)$ such that $\phi(s(e_k))=s(f(e_k))$ for all $s \in S$ and all $k \in K$ .
\end{enumerate}
\end{prop}
\begin{proof} $G$ is not empty since $X_1 \subseteq G$. Let $g \in G$ and suppose that $\phi(s(g))=0$. Then, $s(\phi(g))=0$ and since $s$ is an automorphism, $\phi(g)=0$ which is a contradiction to the definition of $G$. So, $S(G) \subseteq G$. Let $g \in G$. Then, $\phi(g) \neq 0$ and consequently, $\phi(g) \in X_1$. Suppose $s(g)=g$, for some non-identity automorphism $s \in S$. Then, $s(\phi(g))=\phi(s(g))=\phi(g)$. Since $\phi(g) \in X_1$, this contradicts fixedpointfreeness of $S$ on $X_1$. So, $S$ acts without fixed points on $G$. Let $\{e_i | i \in I\}$, $I$ a suitable index set, be a complete set of orbit representatives of the orbits of $S$ acting on $G$. Since $S(X_1) \subseteq X_1$ we know that $X_1$ is a union of orbits of $S$, so there is a subset $J \subseteq I$ such that $X_1=\cup_{j \in J}S(e_j)$. Let $K:=I \setminus J$. Suppose $K$ is not the empty set and let $k \in K$ and $s \in S$. Then, $\phi(s(e_k))=s(\phi(e_k))$. $\phi(e_k) \in \cup_{j \in J} S(e_j)$ and so we can define the function $f:\{e_k | k \in K\} \longrightarrow \cup_{j \in J} S(e_j), e_k \mapsto \phi(e_k)$.
\end{proof}

The construction of the sandwich function $\phi$ becomes especially simple if we let $I=J$, so $G=X_1$ in the language of Proposition \ref{Gabi1}. Let $S$ be a group of automorphisms acting on $\Gamma$ and let $\{e_l| l \in L\}$, $L$ a suitable index set, be the set of orbit representatives  of the action of $S$ on $\Gamma$. Let $\emptyset \neq I \subseteq L$ such that for all $i \in I$, $S$ acts without fixedpoints on $S(e_i)$. We now let $I=J$, so $G=X_1=\cup_{i \in I}S(e_i)$. So, $K=\emptyset$ and the construction of $\phi$ is very easy. For $\gamma \in \Gamma \setminus G$ we have $\phi(\gamma)=0$ and for $ \gamma \in G=X_1$, $\phi(\gamma)=\gamma$. We are now in a position to construct $N:=M_0(X, \Gamma, \phi, S)$. If $N$ is primitive depends on the subgroups contained in $\Gamma \setminus G$ because in the language of Theorem \ref{EEst1}, $\Gamma \setminus G=\Gamma_0$. In particular, $N$ will be $1$-primitive if $\Gamma \setminus G$ is not the union of cosets of any non-trivial subgroup of $\Gamma$. Then, property (P) of Theorem \ref{EEst1} is obviously fulfilled since the set $C$ only contains the trivial group $\{0\}$. For example one can take $S$ as a group of fixedpointfree automorphisms of a group $\Gamma$ and let $G$ be a union of non-zero orbits of $S$ acting on $\Gamma$ such that the size of $\Gamma \setminus G$ is not divisible by the order of any non-trivial subgroup contained in $\Gamma$. By letting $I=J$, so $G=X_1$ in the language of Proposition \ref{Gabi1} and contructing $\phi$ as in Proposition \ref{Gabi1}, we can  be assured that $M_0(X, \Gamma, \phi, S)$ is a $1$-primitive near-ring. Note that if we have $S$ as a group of fixedpointfree automorphism on $\Gamma$, then we may also choose $X_1 = \Gamma \setminus \{0\}$ and obtain $M_0(X, \Gamma, \phi, S) = M_S(\Gamma)$. 

The following easy to establish proposition shows how big those primitive sandwich centralizer near-rings will be. Their size depends on the number of orbits of $S$ acting on $X_1$. For a set $M$ we let $|M|$ be its cardinality.
\begin{prop}\label{GE1} Let $N:=M_0(X, \Gamma, \phi, S)$ be a finite $1$-primitive sandwich centralizer near-ring with $X, \Gamma, \phi, S$ fulfilling the assumptions of Theorem \ref{EEst1}. Let $k$ be the number of orbits of $S$ when acting on $X_1$. Then, $|N|=|\Gamma|^k$. 
\end{prop}
\begin{proof}
Let $\{e_i | i \in I\}$, $I$ a suitable index set, be a complete set of orbit representatives of the orbits of $S$ acting on $X_1$.  Let $k=|I|$. Define a function $f: X \longrightarrow \Gamma$ in the following way: $f(0)=0$, $f(e_i):=\gamma_i$ for $i \in I$, $\gamma_i \in \Gamma$ and $f(s(e_i)):=s(f(e_i))$. Since $S$ acts without fixedpoints on $X_1$ we see that $f$ is well defined and $f \in M_0(X, \Gamma, \phi, S)$. Conversely, any function in $M_0(X, \Gamma, \phi, S)$ is completely determined once one knows its function values on some set of orbit representatives. From that we see that the size of $N$ is $|\Gamma|^{k}$.
\end{proof}

We will use the construction of $\phi$ we obtained in Proposition \ref{Gabi1} to give another example how Theorem \ref{EEst1} can be used to construct $1$-primitive near-rings with $S \neq \{id\}$. We let $(\Gamma,+) := (\Bbb{Z}_{pq},+)$ where $p$ and $q$ are two prime numbers such that $p$ does not divide $q-1$ and $q$ does not divide $p-1$ and $p \neq q$. We take $S=\mathrm{Aut}(\Gamma,+)$. Any group automorphism $s$ of $\Bbb{Z}_{pq}$ is of the form $s: \Gamma \longrightarrow \Gamma, x \mapsto x\cdot a$, where $a \in \Gamma=\Bbb{Z}_{pq}$ is coprime to $pq$ and $\cdot$ is the usual multiplication in $\Bbb{Z}_{pq}$. So, $S$ has $(p-1)(q-1)$ elements, and consequently the orbit $S(1)$ of the number $1 \in \Gamma$ has $(p-1) (q-1)$ elements. Suppose that $s_1, s_2 \in S$, and $s_1(s_2(1))=s_2(1)$. Then, $s_2^{-1}(s_1(s_2(1)))=1$. Since $S$ and $S(1)$ has the same number of elements, $s_2^{-1}\circ s_1 \circ s_2 = id$ and consequently, $s_1 = id$. This means that $S$ acts without fixedpoints on $S(1)$ ($S$ itself is not fixedpointfree on $\Bbb{Z}_{pq}$). We now let, in the notation of Proposition \ref{Gabi1}, $G=X_1:=S(1)$ and $X:=\{0\} \cup X_1$ and define the function $\phi: \Gamma \longrightarrow X$ with $\phi_{|X}=id$ as $\phi(\gamma)=0$ if $\gamma \in \Gamma \setminus G$ and for $\gamma \in G$, $\phi(\gamma)=\gamma$. So, by Proposition \ref{Gabi1} we have for all $\gamma \in \Gamma$ and all $s \in S$, $\phi(s(\gamma))=s(\phi(\gamma))$. Consequently we can build the sandwich centralizer near-ring $M_0(X, \Gamma, \phi, S)$ and it remains to show that property (P) of Theorem \ref{EEst1} is fulfilled. Suppose there is a non-trivial subgroup $I$ of $\Gamma$ such that $\Gamma_0=\Gamma \setminus G$ is a union of cosets of $I$ and therefore also $G$ is a union of cosets of $I$. As a proper subgroup of $\Gamma$, $I$ can only have order $p$ or order $q$. Suppose $I$ has order $p$. Then, $p$ divides the number of elements in $G$ which is $(p-1)(q-1)$. Since $p$ is a prime number, $p$ must divide $q-1$. But this is not the case by choice of the prime numbers $p$ and $q$. The same argument holds if $I$ is assumed to have order $q$. This shows that $\Gamma_0=\Gamma \setminus G$ is not a union of cosets of some non-trivial subgroup of $\Gamma$. Hence, property (P) is fulfilled and $M_0(X,\Gamma, \phi, S)$ is $1$-primitive on $\Gamma$. Note that $\Gamma_0$ contains all the elements of $\Gamma$ which do not have coprime order to $pq$. Thus, any  element in $\Gamma_0$ generates a subgroup of $\Gamma$ which is again contained in $\Gamma_0$. Since $\Gamma_0 \neq \{0\}$, $M_0(X,\Gamma, \phi, S)$ is not $2$-primitive on $\Gamma$ by Theorem \ref{EEst2}. The size of $M_0(X, \Gamma, \phi, S)$ is $pq$ by Corollary \ref{GE1}.

Whenever $X_1$ is just one orbit of $S$ as in the example, then by Corollary \ref{GE1}, $M_0(X,\Gamma, \phi, S)$  has size $\Gamma$. If $S$ is a fixedpointfree automorphism group of $\Gamma$ containing at least two non-identity automorphisms and if $\Gamma$ is finite, then $M_0(X,\Gamma, \phi, S)$  is a so called planar near-ring by Theorem 4.5 of \cite{Wendt2}. Planar near-rings are rich in applications, see \cite{Clay}. Using our main Theorems \ref{EEst1} and \ref{EEst2} and the method of constructing $\phi$ according to Proposition \ref{Gabi1} one could now systematically investigate primitive near-rings acting on special types of groups $\Gamma$. This seems to be an interesting topic for further research but does not lie within the scope of this paper.

\end{document}